\documentclass[11pt,a4paper,leqno]{amsart}
\usepackage{amsfonts}
\usepackage{amsthm}
\usepackage{amsmath}
\usepackage{amscd}
\usepackage[latin2]{inputenc}
\usepackage{t1enc}
\usepackage[mathscr]{eucal}
\usepackage{indentfirst}
\usepackage{graphicx}
\usepackage{graphics}
\usepackage{pict2e}
\usepackage{epic}
\usepackage{lipsum}  
\usepackage{blindtext}
\usepackage{stmaryrd}
\usepackage{authblk}
\usepackage{subcaption}
\usepackage[symbol]{footmisc}
\numberwithin{equation}{section}
\usepackage[margin=2.9cm]{geometry}
\usepackage{epstopdf}
\RequirePackage{doi}
\usepackage{hyperref}
\allowdisplaybreaks
\usepackage{cite}

\usepackage{verbatim,amssymb,amsfonts}
\usepackage{mathrsfs}
\usepackage{mathtools}
\usepackage{tikz-cd}
\usepackage{indentfirst}
\usepackage{slashed}
\usepackage{thmtools}
\usepackage{setspace}
\usepackage{enumerate}
\usepackage{accents}
\theoremstyle{plain}
\newtheorem{theorem}{Theorem}[section]
\newtheorem{lemma}[theorem]{Lemma}
\newtheorem{corollary}[theorem]{Corollary}

\newtheorem{conjecture}{Conjecture}
 \theoremstyle{definition}

\newtheorem{remark}[theorem]{Remark}

\newcommand\norm[1]{\lVert#1\rVert}

\DeclarePairedDelimiterX{\inp}[2]{\langle}{\rangle}{#1, #2}

\newcommand{\cE}{{\mathcal E}}

\newcommand{\spec}{{\mathrm{Spec}}}

\newcommand{\rank}{{\mathrm{rank}}}
\newcommand{\nul}{{\mathrm{null}}}

\newcommand{\ba}{\begin{eqnarray}}
\newcommand{\na}{\end{eqnarray}}
\newcommand{\ban}{\begin{eqnarray*}}
\newcommand{\nan}{\end{eqnarray*}}

\newcommand{\C}{{\mathbb C}}

\newcommand{\R}{{\mathbb R}}

\renewcommand{\thefootnote}{\fnsymbol{footnote}}

\makeatletter
\g@addto@macro{\endabstract}{\@setabstract}
\newcommand{\authorfootnotes}{\renewcommand\thefootnote{\@fnsymbol\c@footnote}}%
\makeatother

\makeatletter
\@namedef{subjclassname@2020}{%
  \textup{2020} Mathematics Subject Classification}
\makeatother

\title[]{Some results on spectrum and energy of \\ graphs with loops}



\begin{document}

\begin{center}
    \vspace{-1cm}
	\maketitle
	
	\normalsize
	\authorfootnotes
	Saieed Akbari\textsuperscript{1}, %\footnote[2]{Corresponding author.},
	Hussah Al Menderj\textsuperscript{2},
%	\footnote{---@usm.my}\textsuperscript{1},
	Miin Huey Ang\textsuperscript{2}, \\ %\footnote{---@gmail.com}\textsuperscript{2}
	Johnny Lim\textsuperscript{2}\footnote[1]{Corresponding author.}, %\thanks{$^*$Corresponding author} 
	Zhen Chuan Ng\textsuperscript{2}
	\par \bigskip
	
	\textsuperscript{1}\small{Department of Mathematical Sciences, Sharif University of Technology, Tehran, Iran} \par
	\textsuperscript{2}\small{School of Mathematical Sciences, Universiti Sains Malaysia, Malaysia}\par \bigskip
	
%	\today
\end{center}

%First author
\address{Department of Mathematical Sciences, Sharif University of Technology, Tehran, Iran}
\email{s\_akbari@sharif.edu}

%Second author

\address{School of Mathematical Sciences, Universiti Sains Malaysia, Malaysia 
}
\email{hussahalmenderj@student.usm.my}
\email{mathamh@usm.my}
\email{johnny.lim@usm.my}
\email{zhenchuanng@usm.my}

\vspace{-0.5cm}

\begin{abstract}
Let $G_S$ be a graph with loops obtained from a graph $G$ of order $n$ and loops at $S \subseteq V(G).$ In this paper, we establish a neccesary and sufficient condition on the bipartititeness of a connected graph $G$ and the spectrum $\spec(G_S)$ and $\spec(G_{V(G)\backslash S})$. We also prove that for every $S\subseteq V(G),$ $\cE(G_S) \geq \cE(G)$ when $G$ is bipartite. Moreover, we provide an identification of the spectrum of complete graphs $K_n$ and complete bipartite graphs $K_{m,n}$ with loops. We characterize any graphs with loops of order $n$ whose eigenvalues are all positive or non-negative, and also any graphs with a few distinct eigenvalues. Finally, we provide some bounds related to $G_S$. %More precisely, we show that if we have $\lambda_1(G_S)$ and $\lambda_{n}(G_{V(G)\backslash S})$ for some $S,$ then the bipartite can be deduced.
\end{abstract}

%\tableofcontents

%%%%%%%%%%%%%%%%%%%%%%%%%%%%%%%%%%%%   
%\vspace{-1em}
\section{Introduction}
\label{intro}
%%%%%%%%%%%%%%%%%%%%%%%%%%%%%%%%%%%%
Let $G$ be a simple graph and $V(G)$ be the set of vertices and $E(G)$ be the set of edges. We call $G$ a graph of order $n$ and size $m$ if $|V(G)|=n$ and $|E(G)|=m$. Denote by $\overline{G}$ the complement of $G$. In this paper, the path of order $n$ is denoted by $P_n$ whereas the complete graph of order $n$ is denoted by $K_n$ and the complete bipartite graph with parts $M$ and $N$ with sizes $m$ and $n,$ is denoted as $K_{m,n}$. 
%See \cite{biggs1993algebraic,wilson2015introduction} for more details on graph theory.
Let $V(G)=\{v_1,\ldots, v_n\}$. For $i=1,2,\ldots,n,$ we denote  $d_i$ the degree of $v_i$ and $\Delta(G)=\max_{1 \leq i \leq n} d_i(G).$ In addition, $G$ is called a \textit{$(a,b)$-semi-regular} if $d_i=a$ or $b$ for $i=1,\ldots, n.$  

The adjacency matrix of $G$, denoted by $A(G)=(a_{ij}),$ whose entry is $a_{ij}=1 $ if $v_i$ and $v_j$ are adjacent and $a_{ij}=0$ otherwise. The characteristic polynomial and eigenvalues of $G$ are the characteristic polynomial and eigenvalues of $A(G)$, respectively. As $A(G)$ is a real symmetric matrix, all eigenvalues of $G$ are real numbers and thus can be ordered as $\lambda_1(G) \geq \lambda_2(G)\geq \cdots \geq \lambda_n(G)$ with $\lambda_1(G)$ and $\lambda_n(G)$ being the largest and the smallest eigenvalue of $G,$ respectively. See \cite{biggs1993algebraic,cvetkovic2010introduction, wilson2015introduction} for more details. All eigenvalues of $G$ with each respective algebraic multiplicity give the spectrum of $G,$ denoted by 
$
\spec(G)= 
\begin{pmatrix} 
\lambda_1 & \lambda_2 & \cdots & \lambda_n \\ 
a_1 & a_2 & \cdots & a_n 
\end{pmatrix},
$ 
where $a_i$ is the algebraic multiplicity of $\lambda_i$. The energy of $G$ is defined as 
\begin{align*}
\cE(G)&=\sum^n_{i=1} |\lambda_i(G)|, \quad (\text{cf. \cite{gutman1979}}).
%\cE(G_S)&=\sum^n_{i=1} \left|\lambda_i(G)-\frac{\sigma}{n}\right|, \quad (\text{cf. \cite{gutman2019energies}}).
\end{align*}
%For more detail about $\cE(G)$, see \cite{gutman1979,gutman2019energies,gutman2019graph,gutman2020research,gutman2020bounds,gutman2021energy,jovanovic2023,ma2019low,mallion1974graphical,mallion1974graph,shelash2020pseudospectrum,zhou2020lower}.

In the early 1970s, graph theory had been found to have an important application in the study of calculation of electron and molecules energy \cite{mallion1974graphical,mallion1974graph}. This has initiated the emerging of the concept of energy of simple graphs in 1978 by Gutman \cite{gutman1979} which greatly advances the research study of graphs and the energy of graphs from 1995 till now \cite{ghodrati2022graph,gutman1979,gutman2019energies,gutman2019graph,gutman2020research,gutman2020bounds,ma2019low,shelash2020pseudospectrum,zhou2020lower}. 

The graph obtained from $G$ by attaching a self-loop at each of the vertex in $S\subseteq V(G)$, is called the \textit{self-loop graph} of $G$ at $S$, denoted by $G_S$. Generalizing from the definition of $A(G)$, $A(G_S)=J_S+A(G)$ where $(J_S)_{i,j}=1$ if $i=j$ and $v_i\in S$, and $(J_S)_{i,j}=0$ otherwise. Similar to $G$, the eigenvalues of $G_S$ are the eigenvalues of $A(G_S).$ It can be verified that all properties of eigenvalues of $A(G)$ are attained by $A(G_S)$. The energy of $G_S$  of order $n$ with $|S|=\sigma$ is defined as 
\begin{align*}
\cE(G_S)&=\sum^n_{i=1} \left|\lambda_i(G_S)-\frac{\sigma}{n}\right|, \quad (\text{cf. \cite{gutman2021energy}}).
\end{align*}
In some occasion, we also denote $G$ as the self-loop graph and $G_0$ is the ordinary graph obtained from $G$ by removing all its self-loops. When $\sigma=n,$ we write $G_S$ as $\widehat{G}.$

Self-loop graphs have been shown to play a significant role in the mathematical study of heteroconjugated molecules \cite{gutman1979topological,gutman1990topological,mallion1974graphical}. Recently in 2022, Gutman et al. have introduced the concept of energy of self-loops graphs for the first time in \cite{gutman2021energy}. The study of energy of self-loop graphs is still very new with results appeared in only two papers \cite{gutman2021energy,jovanovic2023}. In \cite{gutman2021energy}, the following results were proved:

Theorem 1. Let $G$ be a bipartite graph of order $n$, with vertex set $V$. Let $S$ be a subset of $V$. Then, $\cE(G_S)=\cE(G_{V\backslash S})$.

Theorem 2. Let $G_S$ be a self-loop graph of order $n$, with $m$ edges, and $|S|=\sigma$. Let $\lambda_1\geq \lambda_2 \geq \cdots \geq \lambda_n$  be its eigenvalues. Then $\sum^n_{i=1} \lambda_i^2 = 2m+\sigma.$

Theorem 3. Let $G_S$ be a self-loop graph of order $n$, with $m$ edges, and $|S|=\sigma$. Then
\[
\cE(G_S)\leq \sqrt{n\left(2m+\sigma -\frac{\sigma^2}{n}\right)}.
\]

This paper consists of four sections of main results. Section 2 first completely determines the $\spec((K_n)_S)$ and $\spec((K_{m,n})_S)$ for all $n,m \geq 1$ using Theorem 2. The results on $\spec((K_n)_S)$ are then used to completely characterize those $G_S$ with only positive or non-negative eigenvalues as well as with few distinct eigenvalues. In Section 3, a necessary and sufficient result on each $\lambda_i(G_{V\backslash S}) = 1-\lambda_i(G_S)$  in relation with the respective $\lambda_i(G_S)$ for every $i=1,\ldots,n$  when $G$ is bipartite, is obtained and this gives a simplified proof of Theorem 1. We also show that $\cE(G_S)\geq \cE(G)$ when $G$ is bipartite and a conjecture is given. In Section 4, an alternative proof of Theorem 3 is given using Cauchy-Schwarz inequality. This approach leads to a result on the semi-regular graphs with self-loops. An upper bound for $\lambda_1(G_S)$  analogous to a classical bound for $\lambda_1(G)$  in terms of  $\Delta(G)$ is obtained. The existence of a certain semi-regular graphs is shown when the upper bound for $\lambda_1(G_S)$ is attained. 

%%%%%%%%%%%%%%%%%%%%%%%%%%%%%%%%%%%%%%%%%%
\section{Some characterization of self-loop graphs by its eigenvalues}
\label{sec2}
%%%%%%%%%%%%%%%%%%%%%%%%%%%%%%%%%%%%%%%%%%

In this section, we aim to provide some characterization of self-loop graphs with positive and non-negative eigenvalues, as well as those with few distinct eigenvalues. Before that, an identification of  $\spec((K_n)_S)$ and $\spec((K_{m,n})_S)$ are needed for arbitrary $S.$ The results we obtained are a generalisation of the classical spectrum result of $\spec(K_n)$ and $\spec(K_{m,n})$ when $\sigma=0.$

The self-loop spectrum characterization for both $K_n$ and $K_{m,n}$ are technical and require careful examinations of several cases. %The graphs $K_n$ and $K_{m,n}$ are known to be two families of graphs with adjacency matrix of low ranks. This implies the determinant of their adjacency matrix is zero and thus $0$ is one of the eigenvalues.
As a rule of thumb, consider $G$ as $K_n$ or $K_{m,n},$ then the steps to determine $\spec(G_S)$ are as follows. \\
%\bigskip

\textbf{Step 1. Determination of the multiplicity of eigenvalue $0$.}

%Consider the adjacency matrix $A(G_S).$ 
Determine the rank of $A(G_S)$ through each of its submatrices. By the Rank-Nullity Theorem, this implies the nullity of $A(G_S),$ the multiplicity $m_0$ of the eigenvalue 0. \bigskip

\textbf{Step 2. Determination of multiplicity of eigenvalue $1$ (for $(K_n)_S$) or $-1$ (for $(K_{m,n})_S$).}

Repeat Step 1 for the matrix $A(G_S)-I_n$ or $A(G_S)+I_n$ to obtain the multiplicity $m_1$ of the eigenvalue $1$ or $-1,$ respectively. \bigskip

\textbf{Step 3. Determination of the remaining eigenvalues and its multiplicities.} 

Find the remaining $n-m_0-m_1$ many eigenvalues by appropriate methods. If there are only two eigenvalues $\lambda_1(G_S),\lambda_2(G_S)$ left, then they can be obtained by solving the simultaneous equations in Lemma~\ref{lambdasum} below. %\bigskip

\begin{lemma}\cite{gutman2021energy}
\label{lambdasum}
Let $G_S$ be a self-loop graph of order $n$ and $m$ edges. Let $\lambda_1(G_S),\ldots,\lambda_n(G_S)$ be its eigenvalues. Then,
\begin{enumerate}[(i)]
\item $\displaystyle \sum^n_{i=1} \lambda_i(G_S)=\sigma,$
\item $\displaystyle \sum^n_{i=1} \lambda^2_i (G_S)=2m+\sigma.$
\end{enumerate}
\end{lemma}
\noindent For $(K_{m,n})_S,$ there are occasions where we need to determine three eigenvalues $\lambda_i((K_{m,n})_S),$ $i=1,2,3.$  In particular, an additional formula of $\sum_i \lambda^3_i((K_{m,n})_S)$ will be introduced, see Lemma~\ref{3closedwalks}.

Without loss of generality, let $S=\{1,2,\ldots, \sigma\}.$ Let $J_{k\times \ell}$ denote the $k \times \ell$ matrix whose all entries are 1 and $j_k$ be the $k \times 1$ vector with all entries 1.

%%%%%%%%%%%%%%%%%%%%%%%%%%%%%%%%%%%%%%%%%%
\subsection{Identification of $\spec((K_n)_S)$}
\label{sec2.1}
%%%%%%%%%%%%%%%%%%%%%%%%%%%%%%%%%%%%%%%%%%

In this subsection, $\spec((K_n)_S)$ for arbitrary $S$ are identified according to $\sigma.$

\begin{theorem}
\label{thm1}
Let $(K_n)_S$ be the self-loop graph of $K_n$ and $|S|=\sigma.$ Then,  $\spec((K_n)_S)$ are determined by the following three cases:
\end{theorem}

\textbf{Case 1.} $\sigma=0.$ Then, $A((K_n)_S)=A(K_n).$ This is the classical case where 
\[
\spec((K_n)_S) =\spec(K_n) = 
\begin{pmatrix}
n-1 & -1 \\
1 & n-1
\end{pmatrix}.
\]

\textbf{Case 2.} $0<\sigma<n.$ Then, 
\[
A((K_n)_S)= 
\left[
\renewcommand{\arraystretch}{1.6}
\begin{array}{c|c}
J_{\sigma \times \sigma} & J_{\sigma \times (n-\sigma)} \\
\hline
J_{(n-\sigma) \times \sigma} & J_{(n-\sigma) \times (n-\sigma)} - I_{n-\sigma}
\end{array}
\right] 
= \left[
\renewcommand{\arraystretch}{1.6}
\begin{array}{c}
B \\
\hline
C
\end{array}
\right].
\]
Clearly, $\rank(B)=1.$ On the other hand, $J_{(n-\sigma)\times(n-\sigma)} - I_{n-\sigma}$ is the adjacency matrix of $K_{n-\sigma},$ which follows from Case 1 there is no zero eigenvalue, thus it is invertible. Thus, all rows in $C$ are linearly independent. We claim that all rows of $C$ are linearly independent to $j_n^T$, a row of $B$.
Let $\alpha_i$ be the $i$-th row of the adjacency matrix of $K_{n-\sigma}.$ Then, $\{\alpha_1,\ldots, \alpha_{n-\sigma}\}$ is a basis for $\R^{n-\sigma}.$ Clearly, $j^T_{n-\sigma}=\sum^{n-\sigma }_{i=1}\frac{1}{n-\sigma-1}\alpha_i$ and this implies that no rows of $B$ is a linear combination of rows of $C.$ 
%Consider the standard basis $\{e_1,\ldots, e_n\}$ of $\R^n,$ then any vector $x\in \R^n$ can be uniquely represented as $x=\sum a_i e_i.$ From the first $\sigma$ columns, one deduces that $a_i=1/n$ for all $1 \leq i \leq n.$ But, $c_{\sigma+i,\sigma+i}$ are all zero. Contradiction.
Thus, 
\[
\rank(A(K_n)_S) = n-\sigma +1, \quad \nul(A(K_n)_S)= \sigma -1.
\]
This completes Step 1. 

Now, we consider 
\[
A(K_n)_S + I_n 
=
\left[
\renewcommand{\arraystretch}{1.6}
\begin{array}{c|c}
J_{\sigma \times \sigma}+ I_\sigma & J_{\sigma \times (n-\sigma)} \\
\hline
J_{(n-\sigma) \times \sigma} & J_{(n-\sigma) \times (n-\sigma)} \end{array}
\right] 
= \left[
\renewcommand{\arraystretch}{1.6}
\begin{array}{c}
B \\
\hline
C
\end{array}
\right].
\]
We have $\rank(C)=1.$ Let $B'=J_{\sigma \times \sigma} + I_\sigma.$ Then,
%It suffices to observe that
\[
B'=J_{\sigma\times\sigma}+ I_\sigma 
= 2I_\sigma + A(K_\sigma).
\]
%The spectrum of $B'$ can be obtained by shifting from that of $K_{\sigma},$ i.e. t
The eigenvalues of $B'$ are $\sigma +1$ with multiplicity $1$ and $1$ with multiplicity $\sigma -1$. Hence, $B'$ are invertible. By a similar argument as in Step 1, all rows of $B$ are linearly independent to $j_n^T$, a row of $C$. Thus, we have 
\[
\rank(A(K_n)_S + I_n)=\sigma +1, \quad \nul(A(K_n)_S +I_n) = n-\sigma-1.
\]
This completes Step 2. 

For Step 3, note that  $n-(\sigma-1)-(n-\sigma-1)=2.$ %that is,  only two eigenvalues $\lambda_1((K_n)_S),\lambda_2((K_n)_S)$ left to be determined. 
By Lemma~\ref{lambdasum}, we have 
\begin{align*}
\lambda_1 ((K_n)_S) + \lambda_2 ((K_n)_S) &= n-1, \\
\lambda^2_1 ((K_n)_S) + \lambda^2_2 ((K_n)_S) &= n^2-2n+2\sigma +1. 
\end{align*}
Solving this, we obtain
\[
\lambda_{1,2}((K_n)_S)= \frac{(n-1)\pm \sqrt{(n-1)^2+4\sigma}}{2}.
\]
%each of which has multiplicity 1. 

As a conclusion, we have
\begin{equation}
\spec((K_n)_S)=
\begin{pmatrix}
\frac{(n-1)+ \sqrt{(n-1)^2+4\sigma}}{2} & 0 & -1  &  \frac{(n-1)- \sqrt{(n-1)^2+4\sigma}}{2} \\
1 & \sigma -1 & n-\sigma -1 & 1
\end{pmatrix}.        
\end{equation}

\textbf{Case 3.} $\sigma =n.$ Let $\widehat{K_n}=(K_n)_S.$ Then, $A(\widehat{K_n}) = A(K_n) + I_n.$ Thus, we have
\begin{equation}\label{eq:specknhat}
\spec(\widehat{K_n})=
\begin{pmatrix}
n & 0 \\
1 & n-1
\end{pmatrix}.        
\end{equation}

%%%%%%%%%%%%%%%%%%%%%%%%%%%%%%%%%%%%%%
\subsection{Identification of $\spec((K_{m,n})_S)$}
\label{sec2.2}
%%%%%%%%%%%%%%%%%%%%%%%%%%%%%%%%%%%%%%

In this subsection, $\spec((K_{m,n})_S)$ for arbitrary $S$ are identified according to $\sigma.$ Before that, we first show a crucial lemma.

\begin{lemma}
\label{3closedwalks}
Let $G=K_{m,n}$ with $S=S_M \cup S_N \subseteq V(G)$ with $S_M \subseteq M$ and $S_N \subseteq N.$ Let $|S_M|=\sigma_M,$  $|S_N|=\sigma_N,$ and $|S|=\sigma=\sigma_M+\sigma_N$. If $\lambda_1(G_S), \ldots ,\lambda_n(G_S)$ are the eigenvalues of $G_S,$ then 
\begin{align*}
\sum^n_{i=1} \lambda^3_i(G_S) 
= 3(m \sigma_N + n\sigma_M)+\sigma.
\end{align*}
\end{lemma}

\begin{proof}
By \cite[Lemma 2.5 \& Result 2h]{biggs1993algebraic}, $  \sum^n_{i=1} \lambda^3_i(G_S)$ is the number of closed walks of length 3. Generally, a closed walk of length 3 in $G_S$ is either a triple self-looping or three walks that involve two edges and a loop.  Without loss of generality, we write vertices that have a loop by $\mathring{v},$ vertices without loop by $\bar{v}$, %\textcolor{red}{$e \in E(G)$} %as an edge incident with $\mathring{v}$ or $\bar{v}$, 
and $\ell$ as a loop. A direct counting of total number of closed walks of length 3 in $G_S$ is as follows. 

\textbf{Case 1.} Starting from a $v \in M$ that is incident with $u \in N$ via an edge $e.$
\begin{enumerate}
    \item For $\mathring{v},$ there are only two closed walks of length 3: $\mathring{v} \ell \mathring{v} e u e \mathring{v}$ and $\mathring{v} e u e \mathring{v} \ell \mathring{v}.$ This gives a total of $2n\sigma_M$ closed walks of length 3.
    \item For $\bar{v},$ it has to be incident with $\mathring{u} \in N$. There is only one such walk: $\bar{v} e \mathring{u} \ell \mathring{u}  e \bar{v}.$  This gives a total of $m\sigma_N$ closed walks of length 3.
\end{enumerate}
Thus, totally there are $2n\sigma_M +m\sigma_N$ closed walks of length 3.

\textbf{Case 2.}  Starting from $u \in N$ that is incident with $\mathring{v}$ or $\bar{v} \in M$ via an edge $e.$  Similar to Case 1, there is a total of $2m\sigma_N +n\sigma_M$ closed walks of length 3.

\textbf{Case 3.} Triple self-looping at $\mathring{v} \in M$ or $\mathring{u} \in N.$ There are $\sigma$ triple self-loopings. 

Thus, the total number of closed walks of length 3 is $3(m\sigma_N+n\sigma_M)+\sigma.$   
\end{proof} 

%\begin{remark}
%By the previous proof, without loss of generality, we may henceforth assume that all self-loops will exhaust all vertices of $M$ first before exhausting vertices in $N$. In particular, $\sum \lambda^3_i(G_S)$ is independent of the position of self-loops for all $\sigma$.  
%\end{remark}

We are now ready to characterize the spectrum of $G_S.$

\begin{theorem}\label{thm2}
Let $(K_{m,n})_S$ be the self-loop graph of $K_{m,n}$ for $S \subseteq V(K_{m,n})$ with $|S|=\sigma.$ %For an arbitrary $S,$  
Assume that, if $0<\sigma \leq m,$ all loops are within $M,$ and if $m<\sigma \leq m+n,$ there are $m$ loops in $M$ and $\sigma-m$ loops in $N,$ respectively.   
Then, the spectrum $\spec((K_{m,n})_S)$ are determined by the following five cases:
\end{theorem}

For convenience, we let $G=K_{m,n}$. Recall that the adjacency matrix of $G$ is given by 
\begin{equation}\label{eq:kmn1}
A(G) = 
\left[
\renewcommand{\arraystretch}{1.6}
\begin{array}{c|c}
\textbf{0}_{m} & J_{m \times n} \\
\hline
J_{n \times m} & \textbf{0}_{n} \end{array}
\right]. 
\end{equation}

%%%%%%%%%%%%%%%%%%%%%%%%%%%%%%%%%%%%%%%%
\textbf{Case 1. $\sigma =0$.}
%%%%%%%%%%%%%%%%%%%%%%%%%%%%%%%%%%%%%%%%
It is clear that $A(G_S)=A(G)$ and so 
\[
\spec(G_S)=\spec(K_{m,n})=
\begin{pmatrix}
\sqrt{mn} & 0 & -\sqrt{mn} \\
1 & m+n-2 & 1
\end{pmatrix}.
\]

%%%%%%%%%%%%%%%%%%%%%%%%%%%%%%%%%%%%%%%%
\textbf{Case 2. $0<\sigma < m$.}
%%%%%%%%%%%%%%%%%%%%%%%%%%%%%%%%%%%%%%%%
The adjacency matrix of $G_S$ is
\begin{equation*}
\label{eq:AS1}
A(G_S)= 
\left[
\renewcommand{\arraystretch}{1.6}
\begin{array}{c|c}
J_S & J_{m \times n} \\
\hline
J_{n \times m} &  \textbf{0}_{n} \end{array}
\right] 
= 
\left[
\renewcommand{\arraystretch}{1.6}
\begin{array}{ccc}
I_\sigma & \textbf{0}_{\sigma \times (m-\sigma)} & J_{\sigma \times n}\\
\hline
\textbf{0}_{(m-\sigma) \times \sigma}& 
\textbf{0}_{m-\sigma} & J_{(m-\sigma) \times n} \\
\hline 
J_{n \times \sigma} & J_{n \times (m-\sigma)}  & \textbf{0}_{n} \\
\end{array}
\right] 
= 
\left[
\renewcommand{\arraystretch}{1.6}
\begin{array}{c}
B\\
\hline
C\\
\hline 
D\\
\end{array}
\right].
\end{equation*}

We proceed with Step 1.
It is straightforward to observe that both submatrices $C$ and $D$ have rank 1 due to repeated rows. One verifies that 
%the rows of $B$ are linearly independent of $C$ and $D,$ respectively, and so having full rank $\sigma.$ Thus, 
$\rank(A(G_S))= \sigma +2$ and we obtain the eigenvalue $0$ with multiplicity  $\nul(A)=m+n-\sigma -2.$

For Step 2, consider the matrix $A(G_S) - I_{m+n},$ which can be viewed in three parts as well: 
\begin{equation*}
\label{eq:AS2}
A(G_S)-I_{m+n}= 
\left[
\renewcommand{\arraystretch}{1.6}
\begin{array}{ccc}
\textbf{0}_\sigma & \textbf{0}_{\sigma \times (m-\sigma)} & J_{\sigma \times n}\\
\hline
\textbf{0}_{(m-\sigma) \times \sigma}& 
-I_{m-\sigma} & J_{(m-\sigma) \times n} \\
\hline 
J_{n \times \sigma} & J_{n \times (m-\sigma)}  & -I_{n} \\
\end{array}
\right] 
= 
\left[
\renewcommand{\arraystretch}{1.6}
\begin{array}{c}
B\\
\hline
C\\
\hline 
D\\
\end{array}
\right]. 
\end{equation*}

One checks that $\rank(B)=1,$ $\rank(C)=m-\sigma,$ and $\rank(D)=n.$ 
Thus, $\rank(A(G_S) - I_{m+n})=m+n -\sigma +1.$ So, $G_S$ has the eigenvalue $1$ with multiplicity $\nul(A(G_S)-I_{m+n})=\sigma -1.$ 

For Step 3, there are $m+n-\nul(A(G_S))-\nul(A(G_S)-I_{m+n})=3$  eigenvalues $\lambda_i=\lambda_i(G_S),$ $i=1,2,3,$ yet to be determined. Now, by Lemma~\ref{lambdasum} and Lemma~\ref{3closedwalks}, we obtain 
\begin{align*}
\lambda_1 + \lambda_2 + \lambda_3 &= 1, \\
\lambda^2_1 + \lambda^2_2 + \lambda^2_3 &= 2mn+1, \\
\lambda^3_1 + \lambda^3_2 + \lambda^3_3 &= 3n\sigma +1.
\end{align*}
By a fundamental property of a cubic polynomial of $\lambda_i,$ $i=1,2,3,$ we write 
\begin{equation*}\label{eq:plambda1} 
p(\lambda)
=(\lambda- \lambda_1)(\lambda- \lambda_2)(\lambda- \lambda_3)  
=\lambda^3+a\lambda^2+b\lambda+c. 
\end{equation*} 
Solving the coefficients yield 
\begin{align*}\label{eq:plambda2} 
a&%= -\sum^3_{i=1}\lambda_i
=-1, \quad
b%= \sum_{i\neq j}\lambda_i\lambda_j
= \frac{1}{2}\left(\left(\sum^3_{i=1}\lambda_i\right)^2 - \sum^3_{i=1}\lambda_i^2\right)= -mn, \\
c&= \frac{(\sum^3_{i=1}\lambda_i)^3 - (\sum^3_{i=1}\lambda_i^3) - 3(\sum^3_{i=1}\lambda_i)(\sum_{i\neq j}\lambda_i\lambda_j) }{3} = n(m-\sigma). \nonumber
\end{align*}  

%for which all coefficients are explicitly computed above. 
Thus, $\lambda_1,\lambda_2,\lambda_3$ are exactly the roots of $p(\lambda)=\lambda^3-\lambda^2-mn\lambda + n(m-\sigma).$ Note that $p(0)>0,$ $p(1)=-n\sigma <0,$ $\lim_{\lambda\to -\infty} p(\lambda)= -\infty,$ and $\lim_{\lambda \to +\infty}p(\lambda) = +\infty.$ Thus, by Intermediate Value Theorem, the three roots of $p(\lambda)$ are in the intervals $(-\infty,0), (0,1),$ and $(1,+\infty),$ respectively. Hence, $\lambda_1,\lambda_2,$ and $\lambda_3$ are all distinct with multiplicity 1.

%Observe that for $p(\lambda)=\lambda^3-\lambda^2-mn\lambda+n(m-\sigma)$ with $n,\sigma>1$ and $\sigma<m,$  we have
%\[
%\lim_{\lambda\to -\infty} p(\lambda) =-\infty,\quad p(-1)=-2+n\sigma > 0, \quad  p(0)=n(m-\sigma)>0, \quad \lim_{\lambda\to \infty} p(\lambda) =+\infty.
%\]
%Thus, by Rolle's Theorem, there are distinct roots in the intervals $(-\infty,1], [-1,0],$ and $[0,+\infty)$ respectively, each with multiplicity one.  

As a conclusion, when $0<\sigma<m,$ we have
%\begin{theorem}
%When $\sigma<m,$ we have %the spectra of $(K_{m,n})_S$ is
\[
\spec(G_S) = 
\begin{pmatrix}
1 & 0 &  \lambda_1 & \lambda_2 & \lambda_3\\
\sigma -1 & m+n-\sigma -2 & 1 & 1 & 1
\end{pmatrix}
\] 
where $\lambda_1,\lambda_2,\lambda_3$ are the roots of  $p(\lambda)=\lambda^3-\lambda^2-mn\lambda+n(m-\sigma).$
%\end{theorem}

%%%%%%%%%%%%%%%%%%%%%%%%%%%%%%%%%%%%%%%%
\textbf{Case 3. $\sigma = m$.}
%%%%%%%%%%%%%%%%%%%%%%%%%%%%%%%%%%%%%%%%
From \eqref{eq:AS1}, we have 
\[
A(G_S)= 
\left[
\renewcommand{\arraystretch}{1.6}
\begin{array}{c|c}
I_{m} & J_{m \times n}\\
\hline
J_{n \times m} &  \textbf{0}_{n}
\end{array}
\right]. 
\]
%Since the bottom $n$-rows of $A_S$ are all repeated,
Clearly, we have $\rank(A(G_S))=m+1$ and $\nul(A(G_S))=n-1.$ Hence, $G_S$ has the eigenvalue $0$ with multiplicity $n-1.$ Step 1 is complete.

Next, we consider 
\[
A(G_S)-I_{m+n} = 
\left[
\renewcommand{\arraystretch}{1.6}
\begin{array}{c|c}
\textbf{0}_{m} & J_{m \times n}\\
\hline
J_{n \times m} &  -I_{n}
\end{array}
\right].
\]
Similarly, $\rank(A(G_S)-I_{m+n})=n+1$ and $\nul(A(G_S)-I_{m+n})=m-1.$ Thus, $G_S$ has the eigenvalue $1$ with multiplicity $m-1.$ Step 2 is now complete.

For Step 3, there are only $(m+n)-(n-1)-(m-1)=2$ eigenvalues $\lambda_i=\lambda_i(G_S),$ $i=1,2,$ left to be determined. By Lemma~\ref{lambdasum}, we have
\begin{align*}
\lambda_1+\lambda_2 &= 1 \\
\lambda^2_1+\lambda^2_2 &= 2mn+1.
\end{align*}
%The associated characteristic polynomial is 
%\[
%p(\lambda) = (\lambda-\lambda_1)(\lambda-\lambda_2)=\lambda^2-\lambda -mn.
%\]
%Thus, solving this we get
So, we have
\[
\lambda_{1,2} = \frac{1 \pm \sqrt{1+4mn}}{2}.
\]
%We have proved the following statement.
%\vspace{-0.5cm}
%\begin{theorem}
As a conclusion, when $\sigma=m,$ we have %the spectra of $G_S$ are
\[
\spec(G_S)
=
\begin{pmatrix}
\frac{1+\sqrt{1+4mn}}{2}& 1 & 0 &
\frac{1-\sqrt{1+4mn}}{2}  \\
1 & m-1 & n-1 & 1
\end{pmatrix}.
\]
%\end{theorem}

%\bigskip
%%%%%%%%%%%%%%%%%%%%%%%%%%%%%%%%%%%%%%%%
\textbf{Case 4. $m<\sigma<m+n$.}
%%%%%%%%%%%%%%%%%%%%%%%%%%%%%%%%%%%%%%%%
In this case, the adjacency matrix is 
\begin{equation*}
%\label{eq:AS1}
A(G_S)= 
\left[
\renewcommand{\arraystretch}{1.6}
\begin{array}{ccc}
I_m & J_{m \times (\sigma-m)} & J_{m \times (m+n-\sigma)}\\
\hline
J_{(\sigma-m) \times m}& 
I_{(\sigma-m)} & \textbf{0}_{(\sigma-m) \times (m+n-\sigma)} \\
\hline 
J_{(m+n-\sigma)\times m} & \textbf{0}_{(m+n-\sigma) \times (\sigma-m)}  & \textbf{0}_{(m+n-\sigma)} \\
\end{array}
\right] 
= 
\left[
\renewcommand{\arraystretch}{1.6}
\begin{array}{c}
B\\
\hline
C\\
\hline 
D\\
\end{array}
\right].
\end{equation*}

For Step 1, it is clear that $\rank(D)=1$ and both $B$ and $C$ have full rank. Thus, $\rank(A(G_S))=\sigma+1$ and $G_S$ has the eigenvalue $0$ with multiplicity $m+n-\sigma -1.$ For Step 2, consider the matrix
\begin{equation*}
%\label{eq:AS1}
A(G_S)-I_{m+n}= 
\left[
\renewcommand{\arraystretch}{1.6}
\begin{array}{ccc}
\textbf{0}_m & J_{m \times (\sigma-m)} & J_{m \times (m+n-\sigma)}\\
\hline
J_{(\sigma-m) \times m}& 
\textbf{0}_{(\sigma-m)} & \textbf{0}_{(\sigma-m) \times (m+n-\sigma)} \\
\hline 
J_{(m+n-\sigma)\times m} & \textbf{0}_{(m+n-\sigma) \times (\sigma-m)}  & -I_{(m+n-\sigma)} \\
\end{array}
\right] 
= 
\left[
\renewcommand{\arraystretch}{1.6}
\begin{array}{c}
B\\
\hline
C\\
\hline 
D\\
\end{array}
\right].
\end{equation*}
We have $\rank(B)=\rank(C)=1$ and $D$ has full rank. Thus, $\rank(A(G_S))=m+n-\sigma +2$ and $G_S$ has the eigenvalue 1 with multiplicity $\sigma-2.$

For Step 3, there are $(m+n)-(m+n-\sigma-1)-(\sigma-2)=3$  eigenvalues $\lambda_i=\lambda_i(G_S),$ $i=1,2,3,$ left to be determined. Proceed similarly as in Case 2, we apply Lemma~\ref{lambdasum} and Lemma~\ref{3closedwalks} to get
\begin{align*}
\lambda_1 + \lambda_2 + \lambda_3 &= 2, \\
\lambda^2_1 + \lambda^2_2 + \lambda^2_3 &= 2mn+2, \\
\lambda^3_1 + \lambda^3_2 + \lambda^3_3 &= 3m(\sigma + n-m) +2.
\end{align*}
The corresponding cubic polynomial is $p(\lambda)=\lambda^3-2\lambda^2+(1-mn)\lambda+m(m+n-\sigma).$ Using the method as in Case 2, these eigenvalues are distinct, each has multiplicity 1. 
%\begin{theorem}
As a conclusion, when $m<\sigma<m+n,$ we have
\[
\spec(G_S) = 
\begin{pmatrix}
1 & 0 & \lambda_1 & \lambda_2 & \lambda_3\\
\sigma-2 & m+n-\sigma-1  & 1 & 1 & 1
\end{pmatrix},
\] 
where $\lambda_1,\lambda_2,\lambda_3$ are the roots of  $p(\lambda)=\lambda^3-2\lambda^2+(1-mn)\lambda+m(m+n-\sigma).$
%\end{theorem}

%\bigskip
%%%%%%%%%%%%%%%%%%%%%%%%%%%%%%%%%%%%%%%%
\textbf{Case 5. $\sigma=m+n$.}
%%%%%%%%%%%%%%%%%%%%%%%%%%%%%%%%%%%%%%%%
In this case, the adjacency matrix takes the form $A(G_S)= A(G)+I_{m+n}.$ Thus, its spectrum can be obtained by shifting from Case 1:
\[
\spec(G_S)=
\begin{pmatrix}
1+\sqrt{mn} & 1 & 1- \sqrt{mn} \\
1 & m+n-2 & 1
\end{pmatrix}.
\]

%%%%%%%%%%%%%%%%%%%%%%%%%%%%%%%%%%%%   

\bigskip
Now, we are ready to prove the main results of this section. Recall that 
if the eigenvalues are in non-increasing order, then we have the Courant-Weyl Inequality:
\begin{theorem}\cite[Theorem 1.3.15]{cvetkovic2010introduction}
Let $A$ and $B$ be $n\times n$ real symmetric matrix. Then
\begin{align*}
&\lambda_i(A+B)\leq \lambda_j(A)+ \lambda_{i-j+1}(B) \quad \text{ for } n\geq i\geq j\geq 1,\\
&\lambda_i(A+B)\geq \lambda_j(A)+ \lambda_{i-j+n}(B) \quad \text{ for } 1\leq i\leq j\leq n.
\end{align*}
\end{theorem}

By choosing $i=j=n$ in the second inequality above, we obtain the inequality $\lambda_{\min}(A+B)\geq \lambda_{\min}(A)+ \lambda_{\min}(B)$ which is essential in the proof of our next theorem.
\begin{theorem}
Let $G$ be a self-loop graph of order $n$ and has eigenvalues $\lambda_1(G) \geq \lambda_2(G)\geq\cdots\geq \lambda_n(G)$.
\begin{itemize}
    \item[(i)] If $\lambda_i(G)>0$, for $i=1,2,\ldots,n$, then $G$ is disjoint unions of $n$ $\widehat{K_1}$.
    \item[(ii)] If $\lambda_i(G)\geq 0$, for $i=1,2,\ldots,n$, then every connected component of $G$ is either $K_1$ or $\widehat{K_r}$ for some $r\in\{1,2,\ldots,n\}$.
\end{itemize}
\end{theorem}
\begin{proof}\phantom{aaa}
\begin{itemize}
    \item[(i)] Let $H$ be a connected component of $G$. We shall show that $H_0$ is a complete graph. Suppose on the contrary that $H_0$ is not a complete graph. Thus, $H_0$ contains $P_3$ as a vertex-induced subgraph. Note that $\lambda_{\min}(P_3)=-\sqrt{2}.$ So, by Interlacing Theorem (cf. \cite[Cor 1.3.12]{cvetkovic2010introduction}), we obtain $\lambda_{\min}(H_0)\leq -\sqrt{2}$. Note that $J_S=A(H)-A(H_0)$. By the Courant-Weyl Inequalities, we have $\lambda_{\min}(H_0)\geq \lambda_{\min}(H)+\lambda_{\min}(-J_S)$. This implies that 
    \[
    \lambda_{\min}(G) \leq \lambda_{\min}(H)\leq \lambda_{\min}(H_0)-\lambda_{\min}(-J_S)\leq -\sqrt{2}+1,
    \]
%   as the spectrum of $-J_0$ consists of $0$ and $-1$ only. Thus, the spectrum of $G$ contains $\lambda_{\min}(H)<0$ as $H$ is a connected component of $G$, 
    this is a contradiction because $G$ has only positive eigenvalues. Therefore, $H_0$ is a complete graph. By Theorem~\ref{thm1}, $H=\widehat{K_1}$ and we are done.
    
    \item[(ii)] With a slight modification of the proof in Part (i), we can deduce that if $H$ is a connected component of $G$, then $H_0$ is a complete graph. Now, by Theorem~\ref{thm1}, $H=K_1$ or $H=\widehat{K_r}$ for some $r\in\{1,2,\ldots,n\}$. The proof is complete.
\end{itemize}
\end{proof}
%%%%%%%%%%%%%%%%%%%%%%%%%%%%%%%%%%%%   
%\vspace{-1em}
%\textbf{Application 2: Characterization of graphs with few distinct eigenvalues.}
%\label{sec5}    
%%%%%%%%%%%%%%%%%%%%%%%%%%%%%%%%%%%%%%%%%%%%%%%    

In the following, we characterize self-loop graphs of order $n$ with a few distinct eigenvalues.
%using some properties of graphs with loops.

\begin{theorem}
Let $G$ be a self-loop graph of order $n$. Then,
\begin{enumerate}[(i)]
    \item $G$ has exactly one eigenvalue if and only if $G=\overline{K}_n$ or every connected component of $G$ is $\widehat{K_1}.$ 
    \item $G$ has exactly two distinct eigenvalues if and only if all connected components of $G$ are the same and each is $\widehat{K_r}$ for some $r,$ or each component of $G$ is either $K_1$ or $\widehat{K_r}.$
\end{enumerate}
\end{theorem}    

\begin{proof}
\begin{enumerate}[(i)]
    \item %Let $S\subseteq V(G)$ with $|S|=\sigma,$ and $G_S$ be the self-loop graph of $G.$ 
    Assume that $G$ has $\sigma$ loops and $a=\lambda_1(G)=\cdots =\lambda_n(G).$ Then, $\sigma=\mathrm{Tr}(A(G))=na.$  This implies that $a=\displaystyle \frac{\sigma}{n}.$ Since $a$ is a root of a monic polynomial with integer coefficients, it follows that $a$ is an integer. Since $0 \leq \sigma\leq n,$  we have $\sigma=0$ or $\sigma=n.$ %as  $n | \sigma.$ 
    \begin{enumerate}[(a)]
        \item If $\sigma=0,$ then $a=0.$ This implies $A(G)=\textbf{0}_n.$ Thus, $G=\overline{K}_n.$
        \item If $\sigma=n,$ then $a=1.$ So, $A(G)-I_n$ is the adjacency matrix of an ordinary graph with only eigenvalue $0.$ It follows that  $A(G)-I_n=\textbf{0}_n,$ that is, $A(G)=I_n$ and each component of $G$ is $\widehat{K_1}.$
    \end{enumerate} 
    \item  Let $A=A(G).$ If $G$ has two eigenvalues $\lambda_1$ and $\lambda_2,$ then the minimal polynomial of $A$ has the form 
    \[
    p(\lambda)=(\lambda-\lambda_1)(\lambda-\lambda_2).
    \]
    So, we have 
    \begin{equation}\label{eq:Asquare0}
    A^2-(\lambda_2+\lambda_2)A+\lambda_1\lambda_2I_n=\textbf{0}_n.
    \end{equation}
    Let $H$ be a connected component of $G.$ We show that $H_0$ is a complete graph. If $H_0$ is not a complete graph, then there is a vertex induced path of order 3 in $H_0$ such as $\widehat{P}_3:$ $v_rv_sv_k.$

Now, we have $(A^2)_{rk} > 0$ with $A_{rk}=0.$ This is a contradiction because the $(r,k)$-entry of \eqref{eq:Asquare0} is positive on the left, but the $(r,k)$-entry on the right side is 0. Thus, $H_0$ is a complete graph.

By Theorem~\ref{thm1}, $H=K_r$ or $H=\widehat{K_r}$ for some $r.$ If one of the components of $G$ is $K_r$ (resp. $\widehat{K_r}$), then the other components of $G$ should also be $K_r$ (resp. $\widehat{K_r}$) because otherwise $G$ has more than two distinct eigenvalues. If $r=1,$ then $G$ is union of finitely many $K_1$ or $\widehat{K_r}.$ The opposite direction is obvious and the proof is complete.
%by our result in Section 2.1. The proof is complete.
\end{enumerate}
\end{proof}

%%%%%%%%%%%%%%%%%%%%%%%%%%%%%%%%%%%%%%%%%%%%%%
\section{Bipartite graphs and eigenvalues of self-loop graphs}
\label{sec3}    
%%%%%%%%%%%%%%%%%%%%%%%%%%%%%%%%%%%%%%%%%%%%%%
   
%\begin{theorem}[Gutman]
%Let $G$ be a bipartite graph of order $n$ with vertex set $V.$ Let $S \subset V.$ Then, $E(G_s)=E(G_{V\backslash S}).$
%\end{theorem}

%In this paper, we aim to prove the following main theorem, for which the proofs (both directions) will be discussed separately in Theorem~\ref{result1} and \ref{result2} respectively..

%\begin{theorem}\label{mainthm1}
%Let $G$ be a graph and $S \subseteq V(G).$ Let $G_S$ and $G_{V(G)\backslash S}$ be the extended graphs of $G$ with self-loops at $S$ and $V(G)\backslash S$ respectively. Then, the eigenvalues of $G_{V(G) \backslash S}$ are 
%\[
%1-\lambda_1(G_S),...,1-\lambda_n(G_S)
%\] 
%if and only if $G$ is bipartite.
%\end{theorem}

In this section, we provide a necessary and sufficient condition for a graph $G$ being bipartite according to the eigenvalues of $G$ and $G_{V(G)\backslash S}$ for arbitrary $S \subseteq V(G).$ First, we give an observation on similarity.

\begin{lemma}
Let $G$ be a bipartite graph of part sizes $m$ and $n$ and $|S|=\sigma.$ Then, $J_S + A(G)$ is similar to the matrix $J_S-A(G).$
\end{lemma}
\begin{proof}
%Let $G=(M,N)$ with $|M|=m$ and $|N|=n.$ Then, o
One can easily see that if
%\begin{equation}\label{eq:matrixP}
\[
P=
\begin{bmatrix}
I_m  & \textbf{0} \\
\textbf{0} & -I_n 
\end{bmatrix},
\]
%\end{equation} 
then, by \eqref{eq:kmn1},
%\begin{align*}
$
P(J_S + A(G))P^{-1} 
= P(J_S + A(G))P = J_S - A(G).
$
%\end{align*}
Thus, $J_S + A(G)$ is similar to $J_S - A(G).$
\end{proof}
%= PJ_SP + PA(G)P \\

\begin{theorem}
\label{result1}
Let $G$ be a bipartite graph of order $n$ and $|S|=\sigma.$ Then, $1 - \lambda_n(G_S) \geq \cdots \geq 1 - \lambda_1(G_S)$ are the eigenvalues of $G_{V(G)\backslash S},$  where $\lambda_1(G_S)\geq \cdots \geq \lambda_n(G_S)$ are the eigenvalues of $G_S.$
\end{theorem}

\begin{proof}
Note that $A(G_S) = J_S + A(G)$ and $A(G_{V(G)\backslash S}) = J_{V(G) \backslash S} + A(G).$ Hence, we find
%\begin{align*}
\[
J_{V(G) \backslash S} - A(G)
    = I_n -(J_S + A(G)) 
    = I_n - A(G_S).
\]
%\end{align*}
Thus, we have $1 -\lambda_i(G_S),$  $i=1, \ldots,n,$ the eigenvalues of $J_{V(G)\backslash S}- A(G),$ which coincide with the eigenvalues of $A(G_{V(G)\backslash S})$ by similarity in the previous lemma.  
\end{proof}

\begin{theorem}
\label{result2}
Let $G$ be a connected graph of order $n$. Let $G_S$ be its self-loop graph with eigenvalues $\lambda_1(G_S) \geq \cdots \geq \lambda_n(G_S)$ and $S \subseteq V(G).$ Then, the eigenvalues of $G_{V(G)\backslash S}$ are $1-\lambda_n(G_S) \geq \cdots \geq 1-\lambda_1(G_S)$ if and only if $G$ is bipartite. 
\end{theorem}

\begin{proof}
By Theorem~\ref{result1}, it suffices to prove that if the spectrum of $G_{V(G)\backslash S}$ is $1-\lambda_n(G_S) \geq \cdots \geq 1-\lambda_1(G_S),$ then $G$ is bipartite. %To do so, we deduce a condition from $1-\lambda_n(G_S)\geq \cdots \geq 1-\lambda_1(G_S)$ that will give a contradiction if $G$ is not bipartite.

Since $1-\lambda_1(G_S)$ is the smallest eigenvalues of $G_{V(G) \backslash S},$ it follows that the largest eigenvalue of the matrix $I_n-A(G_{V(G)\backslash S})$ is $\lambda_1(G_S).$ Also note that $I_n-A(G_{V(G)\backslash S})=I_n-(J_{V(G)\backslash S}+A(G))=J_S-A(G)$. Thus, we have 
\begin{align*}
\lambda_1(G_S)
&=\max_{||x||=1} x^T(J_S + A(G))x \\
&=\max_{||x||=1} x^T(I_n -A(G_{V(G) \backslash S})x\\
&=\max_{||x||=1} x^T(J_S - A(G))x.
\end{align*}

Suppose that $\lambda_1(G_S)=z^T(J_S-A(G))z$ for some $z$ with $\norm{z}=1.$ Define 
\[
|z|^T =
%\begin{bmatrix}
(|z_1|, 
|z_2|,  
\cdots, 
|z_n|).
%\end{bmatrix}.
\]
Then, we have 
\begin{equation*}%\label{eq:inequality1}
\lambda_1(G_S)=z^T(J_S - A(G))z \leq  |z|^T (J_S +A(G)) |z| \leq \lambda_1(G_S).
\end{equation*}
%Note that $\norm{|z|}=1.$ 
Thus, $z^T(J_S-A(G))z=|z|^T(J_S+A(G))|z|.$ This implies that $z^T(-A(G))z=|z|^T(A(G))|z|$, or equivalently
\begin{equation}
\label{eq:equality0}
\sum_{1\leq i,j\leq n}(-a_{ij})z_iz_j=\sum_{1\leq i,j\leq n}a_{ij}|z_i||z_j|.
\end{equation}
Observe that for each $1\leq i,j\leq n$, we have
\begin{equation}
\label{eq:inequality1}
(-a_{ij})z_iz_j\leq a_{ij}|z_i||z_j|.
\end{equation}
Since $\norm{|z|}=1$  and $|z|^T(J_S+A(G))|z|=\lambda_1(G_S),$ we conclude that $|z|$ is an eigenvector corresponding to $\lambda_1(G_S)$ for the matrix $J_S +A(G).$ %otherwise $|z|^T(J_S+A(G))|z|\neq\lambda_1(G_S).$ 
Since $G$ is connected, by Perron-Frobenius Theorem \cite[\S 2]{gantmacher2005applications}, $\lambda_1(G_S)$ is a simple eigenvalue and there exists an eigenvector $\alpha$ for $\lambda_1(G_S)$ whose all entries are positive. Then $|z|$ is a multiple of $\alpha.$ % Since the multiplicity of $\lambda_1$ is one, every eigenvector of $\lambda_1$ is a multiple of $\alpha.$
This implies that %$|z_i| \neq 0$ and thus 
$z_i \neq 0$ for $i=1,\ldots, n.$ Note that if there exist some $i$ and $j$ with $1\leq i,j\leq n$ such that $a_{ij}=1$ and $z_iz_j>0,$ then \eqref{eq:inequality1} becomes a strict inequality. Thus, by taking the summation on both sides over $i,j,$ we obtain a strict inequality that contradicts \eqref{eq:equality0}. We shall show that this contradiction occurs if $G$ is not bipartite.

%Using \eqref{eq:inequality1}, we argue that $G$ is bipartite by contradiction.
Assume $G$ is not bipartite, that is, $G$ contains an odd cycle with vertices $v_1,v_2,\ldots,$ $v_{2k+1}.$ %Then, \eqref{eq:inequality1} gives 
%\[
%z^T(-A(G))z=|z|^TA(G)|z|
%\]
%or equivalently
%\[
%\sum_{1 \leq i,j \leq n} (-a_{ij})z_iz_j = \sum_{1 \leq i,j \leq n} a_{ij}|z_i||z_j|.
%\]
%Observe that for each $1 \leq i,j \leq n,$
%\[
%-a_{ij}z_iz_j \leq a_{ij} |z_i| |z_j|.
%\]
%To obtain a contradiction, it suffices to find $r$ and $s$ such that $-a_{rs}z_rz_s < a_{rs} |z_r||z_s|.$
%Without loss of generality, assume that the vertices of odd cycle are $v_1,\ldots, v_{2k+1}.$
%Since $v_1, \ldots, v_{2k+1}$ is an odd cycle, 
Then, for $i=1,\ldots, 2k+1$, we have $a_{i,r_i}=1,$ where $r_i=i+1\mod{(2k+1)}$. %Assume $z_i$ is associated to $v_i$ for each $1 \leq i \leq 2k+1.$ %Obviously, 
Note that $z_\ell z_{r_{\ell}}$ cannot be negative for every $1\leq \ell\leq 2k+1$ because there are odd distinct vertices. So, there exists an $\ell$ with $1 \leq \ell \leq 2k+1$ such that $z_\ell z_{r_{\ell}}>0.$ %So, $-a_{\ell, \ell+1} z_{\ell} z_{\ell+1} < a_{\ell,\ell+1} |z_\ell| |z_{\ell+1}|,$ which gives a contradiction.
Thus, $G$ is bipartite. The proof is  complete.
\end{proof}

\begin{remark}
Let $S \subseteq V(G).$  Theorem~\ref{result2} provides a way of determining the bipartiteness of a graph directly from the eigenvalues of its self-loop graphs $G_S$ and the eigenvalues of $G_{V(G)\backslash S}.$ Indeed, if we have $\lambda_1(G_S)$ and $\lambda_n(G_{V(G)\backslash S}),$ we can determine whether $G$ is bipartite.
%: if we have the eigenvalues of $G_S$ and take the largest one $\lambda_1(G_S),$ if $1-\lambda_1(G_S)$ is not equal to $\lambda_n(G_{V(G)\backslash S})$, then $G$ is not bipartite. 
\end{remark}

Another immediate consequence of Theorem~\ref{result2} is the following corollary.
%that we obtain a significantly shorter proof for \cite[Theorem 3]{gutman2021energy} on the coincidence of the graph energy between $G_S$ and $G_{V(G)\backslash S}$.

\begin{corollary}\cite[Theorem 3]{gutman2021energy}
\label{result1.1}
Let $G$ be a bipartite graph of order $n$ with vertex set $V(G).$ Let $S$ be a nonempty subset of $V(G).$ Then,  $\cE(G_S)=\cE(G_{V(G)\backslash S}).$
\end{corollary}

Before closing this section, we discuss a case that answers \cite[Conjecture 2]{gutman2021energy}. %on an inequality between the energy of a self-loop graph and its ordinary graph. 
As a motivation, let us consider the case $G=K_{3,3}.$ Using the spectrum determined Section~\ref{sec2.2}, the energy of $(K_{3,3})_S$ is obtained as follows.

\begin{table}[h]
\centering
\begin{tabular}{ |c|c| } 
 \hline
 $\sigma=|S|$ & $\cE((K_{3,3})_S) \approx$   \\ 
 \hline
 0 & 6.0000 \\ 
 \hline
 1 & 7.0690 \\ 
 \hline
 2 & 7.4513 \\ 
 \hline
 3 & 8.0828\\ 
 \hline
 4 & 7.4513 \\ 
 \hline
 5 & 7.0690 \\ 
 \hline
 6 & 6.0000\\ 
 \hline
\end{tabular}
\caption{The energy of $(K_{3,3})_S.$}
\end{table}
From the table, we observe that when there are loops, $(K_{3,3})_S$ has at least the energy of the ordinary graph $K_{3,3}.$ Thus, it is natural to ask \textit{whether the same is true for all bipartite graphs}.
Indeed, we provide an affirmative answer to this question. Before proving this, let us recall an inequality in \cite{AKBARI2020205}.

\begin{theorem}\cite{AKBARI2020205}
\label{energyineq1}
Let $A$ be a Hermitian matrix in the following block form:
\[
A=
\begin{bmatrix}
B & D \\ D^* & C
\end{bmatrix}.
\]
Then, $\cE(A)\geq 2 \cE(D).$
\end{theorem}

%In particular, $B=C=\textbf{0}$ and $D^*=D^T$ in the case of bipartite graphs. 

\begin{theorem}
\label{energyineq2}
If $G$ is a bipartite graph and $S \subseteq V(G),$ then, $\cE(G_S) \geq \cE(G).$
\end{theorem}

\begin{proof}
Let $A(G)=\begin{bmatrix}
\textbf{0} & D \\ D^T & \textbf{0}
\end{bmatrix}$ be its adjacency matrix. Note that 
\[
 \cE(G_S) = \cE(A(G_S) - \frac{\sigma}{n} I_n),
 \]
where $\sigma=|S|$ and $n=|V(G)|.$ By Theorem~\ref{energyineq1},
%\[
%\cE(A(G)) \leq \cE(G_S).
%\]
%On the other hand, 
\[
\cE(G)%=\cE(A(G))
=2\cE(D) \leq \cE(A(G_S)-\frac{\sigma}{n}I_n) = \cE(G_S).
\vspace{-2em}
\] 
%\vspace{-1em}
\end{proof}

Theorem~\ref{energyineq2} answers a special case (bipartite graphs) of \cite[Conjecture 2]{gutman2021energy}.  
%Moreover, we wonder whether there is a graph such that its energy decreases when loops are added. Thus, 
Lastly, we propose the following conjecture.

\begin{conjecture}
For every simple graph $G,$ there exists $S \subseteq V(G)$ such that $\cE(G_S) > \cE(G).$
\end{conjecture}

%%%%%%%%%%%%%%%%%%%%%%%%%%%%%%%%%%%%%%%%
%\vspace{-1cm}
%\section{Additional results on energy bounds for $G_S$.}
%\newpage
\section{Upper bound for the energy of $G_S$.}
\label{sec4}

%%%%%%%%%%%%%%%%%%%%%%%%%%%%%%%%%%%%%%%%
%\subsection{An upper bound for energy of graphs with loops}
%\label{sec4.1}
%%%%%%%%%%%%%%%%%%%%%%%%%%%%%%%%%%%%%%%%
We first recall the following theorem.
\begin{theorem}\cite[Theorem 4]{ghodrati2022graph}
\label{ineqbound1}
For any $(a_{ij})=A \in M_n(\C)$ with eigenvalues $\lambda_1,\ldots, \lambda_n,$ 
\[
\sum^n_{i=1} |\lambda_i| \leq \sum^n_{i=1} \norm{A_i},
\]
where $A_i$ denotes the $i$-th row of $A$ and $\norm{A_i}= \sqrt{a^2_{i1} + a^2_{i2} + \cdots + a^2_{in}}.$ 
\end{theorem}

For $A \in M_n(\C),$ the \textit{energy of a complex matrix} is defined to be 
\[
\cE(A)= \sum^n_{i}|\lambda_i(A)|,
\]
where $\lambda_1(A),\ldots, \lambda_n(A)$ are eigenvalues of $A.$ Thus, $\cE(A)\leq \sum^n_{i=1} \norm{A_i}.$

As a consequence, we obtain an alternative proof of an upper bound for the energy of $G_S$ given by Gutman et al. and discuss the equality case.

\begin{corollary}\cite[Theorem 6]{gutman2021energy}
Let $G_S$ be a self-loop graph of order $n$ and size $m$ with $|S|=\sigma$.  Then, 
\begin{equation}\label{eq:energyupb}
\cE(G_S) \leq \sqrt{n\left(2m+ \sigma - \frac{\sigma^2}{n}\right)}.
\end{equation}
Suppose the equality holds, then $G_S$ is $(a,b)$-semi-regular,  where 
\[
a = \frac{2m}{n} + \frac{3\sigma}{n} - \frac{2\sigma^2}{n^2} -1 , \quad b = \frac{2m}{n} + \frac{\sigma}{n} -\frac{2\sigma^2}{n^2}.
\]
\end{corollary}

\begin{proof}
Let $\displaystyle B=A(G_S) - \frac{\sigma}{n} I_n.$ %Thus, the eigenvalues of $B$ are $\displaystyle \lambda_1(G_S) - \frac{\sigma}{n}, \ldots, \lambda_n(G_S) - \frac{\sigma}{n}.$
By Theorem~\ref{ineqbound1}, we have 
\[
\cE(G_S)= \sum^n_{i=1} \left|\lambda_i(B) \right| =\sum^n_{i=1} \left|\lambda_i(G_S) - \frac{\sigma}{n} \right|  \leq \sum^n_{i=1} \norm{B_i}.
\] 
%By the definition of $\cE(G_S)$,   
%\[
%\cE(G_S)=\sum^n_{i=1} \left|\lambda_i - \frac{\sigma}{n} \right| \leq \sum^n_{i=1} \norm{B_i}.
%\]
Hence, we have  
\begin{align*}%\label{eq:normB1}
\sum^n_{i=1} \norm{B_i} 
&= \sqrt{d_1 + \left(1-\frac{\sigma}{n}\right)^2} + \cdots + \sqrt{d_\sigma + \left(1-\frac{\sigma}{n}\right)^2} \\ 
&\qquad\qquad+ \sqrt{d_{\sigma+1} + \left(\frac{\sigma}{n}\right)^2} + \cdots + \sqrt{d_n + \left(\frac{\sigma}{n}\right)^2}. \nonumber
\end{align*}
By Cauchy-Schwarz inequality, we obtain
\[
\cE(G_S) \leq  \sqrt{n\left(2m+ \sigma - \frac{\sigma^2}{n}\right)}.
\]
If the equality holds, then $d_i+(1-\frac{\sigma}{n})^2=d_{i+1}+(1-\frac{\sigma}{n})^2$ for $1\leq i \leq \sigma-1$ and $d_j+(\frac{\sigma}{n})^2=d_{j+1} + (\frac{\sigma}{n})^2$ for $\sigma+1  \leq i \leq n-1.$ One deduces that $a=d_1=\cdots =d_\sigma$ and $b=d_{\sigma+1} = \cdots =d_{n}.$ Thus, the $(a,b)$-semi-regularity can be obtained by solving the simultaneous equations $2m=\sigma a + (n-\sigma)b$ and $a+(1-\frac{\sigma}{n})^2=b+\frac{\sigma^2}{n^2}.$ 
This completes the proof.
\end{proof}

%&\leq \sqrt{n}\sqrt{(d_1+\cdots + d_\sigma) + \sigma\left(1-\frac{\sigma}{n}\right)^2 + (d_{\sigma+1} +\cdots +d_n) + (n-\sigma)\frac{\sigma^2}{n^2}} \\

%%%%%%%%%%%%%%%%%%%%%%%%%%%%%%%%%%%%   

%\subsection{Analogous classical upper bound for $\lambda_1$ in terms of maximum degree}
\section{Analogous classical upper bound for $\lambda_1$ in terms of maximum degree}
\label{sec5}
%\label{sec4.2}    
%%%%%%%%%%%%%%%%%%%%%%%%%%%%%%%%%%%%%%%%%%%%%%%

It is well-known that if $G$ is a graph, then $\lambda_1(G)\leq \Delta(G)$. In the next theorem, we generalize it to graphs with self-loop graphs by showing that $\lambda_1(G_S) \leq \Delta(G)+1$.

\begin{theorem}
Let $G_S$ be a self-loop graph of order $n$. Then $\lambda_1(G_S)\leq \Delta(G)+1\leq n.$ Moreover, $\lambda_1(G_S)=n$ if and only if $G_S=\widehat{K_n}.$
\end{theorem}
\begin{proof}
Since $A(G_S)$ is a non-negative matrix, by Perron-Frobenius Theorem, there exists an eigenvector $x$ corresponding to eigenvalue $\lambda_1(G_S)$ such that $x^T= 
[x_1, x_2, \ldots, x_n]$ and $x_i \geq 0$ for $i=1,2,\ldots,n.$ 

Assume that $\displaystyle x_r=\max_{1\leq i\leq n}x_i,$ for some $r\in\{1,2,\ldots,n\}$. By considering the $r$-th entry of both sides of $A(G_S)x=\lambda_1(G_S)x$, we have
\begin{equation}\label{eq:Delta1}
x_{i_1}+x_{i_2}+\cdots+x_{i_{d(v_r)}}+\gamma x_r=\lambda_1(G_S) x_r,\quad \gamma\in\{0,1\}.
\end{equation}
This implies that
\[
\lambda_1(G_S) x_r\leq (\Delta (G)+1)x_r.
\]
Hence, $\lambda_1(G_S)\leq \Delta(G)+1$ as $x_r>0$. Since $\Delta(G)\leq n-1,$ $\lambda_1(G_S) \leq n.$ 

If $\lambda_1(G_S)=n$, then Equation~\eqref{eq:Delta1} becomes
\[
x_{i_1}+x_{i_2}+\cdots+x_{i_{d(v_r)}}+\gamma x_r=n x_r,\quad \gamma\in\{0,1\},
\]
with $d(v_r)=\Delta(G)=n-1$, $\gamma=1$, and $x_{i_k}=x_r$ for $k=1,2,\ldots,n-1$. Hence, $A(G_S)=J_{n\times n}$ indicating $G_S=\widehat{K_n}$. Conversely, if $G_S=\widehat{K_n}$, then by \eqref{eq:specknhat}, we have
\[
\text{Spec}(G_S)=
\begin{pmatrix}
n & 0 \\
1 & n-1
\end{pmatrix},
\]
which gives $\lambda_1(G_S)=n$.
\end{proof}

The next theorem provides a sharp lower bound for $\lambda_1(G)$ for any connected graph with loops.

\begin{theorem}
Let $G$ be a connected graph of order $n$ and size $m$. If $S\subseteq V(G)$ with $|S|=\sigma$, 
then
\[
\lambda_1(G_S)\geq \frac{2m}{n}+\frac{\sigma}{n}.
\]
Moreover, if $G$ is a
$\left(k,k+1\right)$-semi-regular graph for some natural number $k,$ 
\[
d_G(v)=\begin{cases}
k,\quad & \text{if $v\in S$,}\\ k+1,\quad &\text{if $v\in V(G)\backslash S,$}  
\end{cases}
\]
where $d_G$ is the degree of vertices of $G$,
then 
\[
\lambda_1(G_S)=\frac{2m}{n}+\frac{\sigma}{n}.
\]
\end{theorem}

\begin{proof}
%Let $A(G)$ be the adjacency matrix of $G$. 
Since $J_S+A(G)$ is a real symmetric matrix, by Rayleigh quotient, we obtain
\begin{align*}
\lambda_1(G_S)&=\max_{0\neq x\in\mathbb{R}^n} \frac{x^T(J_S+A(G))x}{x^Tx}\\
&\geq  \frac{j^T(J_S+A(G))j}{j^Tj}\\
&=\frac{j^T(J_S)j}{j^Tj}+\frac{j^T(A(G))j}{j^Tj}=\frac{2m}{n}+\frac{\sigma}{n}.
\end{align*}

To show the equality, we consider a $\left(k,k+1\right)$-semi-regular graph $G$, $1\leq k\leq n$, such that
\[
d_G(v)=
\begin{cases}k,\quad & \text{if $v\in S$,}\\ k+1,\quad &\text{if $v\in V(G)\backslash S$.}  
\end{cases}
\] 
Then, $A(G_S)=J_S+A(G)$ gives $(J_S+A(G))j=(k+1)j$. This shows that $j$ is an eigenvector corresponding to the eigenvalue $k+1$, and thus $k+1\in \text{Spec}(G_S)$. Notice that for $G,$  
\[
2m=\sum_{i=1}^n d_{G}(v_i)=\sigma k +(n-\sigma)(k+1) = -\sigma + nk + n.
\]
Therefore, we have
\[
\frac{2m}{n}+\frac{\sigma}{n}=k+1\in \text{Spec}(G_S),
\]
and $j$ is an eigenvector corresponding to $\frac{2m}{n}+\frac{\sigma}{n}$. 

%This is equivalent to $2m+\sigma'=kn$ for $k \in \N \backslash \{1\}.$ Thus, $\frac{2m}{n} \leq k \leq \frac{2m}{n}+1$ as $0 \leq \sigma \leq n.$ Let $\sigma'=\lceil \frac{2m}{n} \rceil n -2m.$ Then, it follows that $G_S$ is  $(a,b)$-semi-regular with $a$ and $b$ given by \eqref{eq:absemireg1} such that
%\[
%d(v)=\begin{cases}a,\quad & \text{if $v\in S$,}\\ b,\quad %&\text{if $v\in V(G)\backslash S$.}  \end{cases}
%\]
%Then $A(G_S)=J_S+A(G)$ gives $(J_S+A(G))j=bj$. This shows that $j$ is an eigenvector corresponding to the eigenvalue $b$, and thus \[
%b=\frac{2m}{n} + \frac{\sigma'}{n} \in \text{Spec}(G_S).
%\]
To show that $\frac{2m}{n}+\frac{\sigma}{n}$ is the largest eigenvalue of graph $G_S$, we suppose on the contrary that $\frac{2m}{n}+\frac{\sigma}{n}=\lambda_i(G_S)$ for some $i\geq 2$. Since $A(G_S)$ is a non-negative matrix, by Perron-Frobenius Theorem, there exists an eigenvector $x$ corresponding to eigenvalue $\lambda_1$ such that $x^T=[x_1,\ldots, x_n]$ and $x_i>0$ for $i=1,2,\ldots,n.$ Since $\lambda_i\neq \lambda_1$, it follows that the eigenvector $x$ and $j$ are orthogonal, a contradiction. %which implies
%\[
%\sum_{i=1}^n x_i=x^T \cdot j=0.
%\]
%But $x_i>0$ for all $i=1,2,\ldots,n$, a contradiction. 
Hence, $\lambda_1(G_S)=\frac{2m}{n}+\frac{\sigma}{n}$.
\end{proof}

%\section{Concluding remarks}

%%%%%%%%%%%%%%%%%%%%%%%%%%%%%%%%%%%%
\vspace{0.5cm}
\textbf{Acknowledgement.}  Johnny Lim acknowledges the support from the Ministry of Higher Education Malaysia for Fundamental Research Grant Scheme with Project Code: 
\linebreak FRGS/1/2021/STG06/USM/02/7.  The research of this paper was done while the first author visited the School of Mathematical Sciences, Universiti Sains Malaysia, as a visiting professor; he would like to also thank the institute for the invitation and partial financial support. The research of the first author was supported by grant number G981202 from the Sharif University of Technology. %\\    

%%%%%%%%%%%%%%%%%%%%%%%%%%%%%%%%%%%%%%%%%%%%%%%%%%%%%%%%%
\vspace{0.5cm}
\textbf{Conflicts of interest.} The authors declare no conflict of interest.

\vspace{0.5cm}
%%%%%%%%%%%%%%%%%%%%%%%%%%%%%%%%%%%%%%%%%%%%%%%%%%%%%%%%%

\bibliography{bibliography}{}

\providecommand{\bysame}{\leavevmode\hbox to3em{\hrulefill}\thinspace}
\providecommand{\MR}{\relax\ifhmode\unskip\space\fi MR }
\providecommand{\MRhref}[2]{%
  \href{http://www.ams.org/mathscinet-getitem?mr=#1}{#2}
}
\providecommand{\href}[2]{#2}
\begin{thebibliography}{10}

\bibitem{AKBARI2020205}
S.~Akbari, A.~H. Ghodrati, and M.~A. Hosseinzadeh, \emph{Some lower bounds for
  the energy of graphs}, Linear Algebra and its Applications \textbf{591}
  (2020), 205--214.

\bibitem{biggs1993algebraic}
N.~Biggs, \emph{Algebraic {G}raph {T}heory}, no.~67, Cambridge University
  Press, 1993.

\bibitem{cvetkovic2010introduction}
D.~M. Cvetkovi{\'c}, P.~Rowlinson, and S.~Simi{\'c}, \emph{An {I}ntroduction to
  the {T}heory of {G}raph {S}pectra}, vol.~75, Cambridge University Press
  Cambridge, 2010.

\bibitem{gantmacher2005applications}
F.~R. Gantmacher and J.~L. Brenner, \emph{Applications of the {T}heory of
  {M}atrices}, Courier Corporation, 2005.

\bibitem{ghodrati2022graph}
A.~H. Ghodrati, \emph{Graph {E}nergy and {H}adamard$'$s {I}nequality}, MATCH
  Commun. Math. Comput. Chem. \textbf{87} (2022), no.~3, 673--682.

\bibitem{gutman1979}
I.~Gutman, \emph{The energy of a graph}, Ber. Math. -Statist. Sekt.
  Forschungsz. \textbf{103} (1978), 1--22.

\bibitem{gutman1979topological}
\bysame, \emph{Topological studies on heteroconjugated molecules}, Theoretica
  chimica acta \textbf{50} (1979), no.~4, 287--297.

\bibitem{gutman1990topological}
\bysame, \emph{Topological studies on heteroconjugated molecules. {VI}.
  {A}lternant systems with two heteroatoms}, Zeitschrift f{\"u}r Naturforschung
  A \textbf{45} (1990), no.~9--10, 1085--1089.

\bibitem{gutman2019energies}
I.~Gutman and B.~Furtula, \emph{Energies of {G}raphs--survey, {C}ensus,
  {B}ibliography}, Center for Scientific Research, Kragujevac (2019).

\bibitem{gutman2019graph}
\bysame, \emph{Graph energies and their applications}, Bulletin (Acad{\'e}mie
  serbe des sciences et des arts. Classe des sciences math{\'e}matiques et
  naturelles. Sciences math{\'e}matiques) (2019), no.~44, 29--45.

\bibitem{gutman2020bounds}
I.~Gutman and M.~R. Oboudi, \emph{Bounds on graph energy}, Discrete Math. Lett.
  \textbf{4} (2020), 1--4.

\bibitem{gutman2020research}
I.~Gutman and H.~Ramane, \emph{Research on graph energies in 2019}, MATCH
  Commun. Math. Comput. Chem. \textbf{84} (2020), no.~2, 277--292.

\bibitem{gutman2021energy}
I.~Gutman, I.~Red{\v{z}}epovi{\'c}, B.~Furtula, and A.~Sahal, \emph{Energy of
  {G}raphs with {S}elf-{L}oops}, MATCH Commun. Math. Comput. Chem. \textbf{87}
  (2021), 645--652.

\bibitem{jovanovic2023}
I.~Jovanovi{\'c}, E.~Zogi{\'c}, and E.~Glogi{\'c}, \emph{On the conjecture
  related to the energy of graphs with self-loops}, MATCH Commun. Math. Comput.
  Chem. \textbf{89} (2023), 479--488.

\bibitem{ma2019low}
X.~Ma, \emph{A low bound on graph energy in terms of minimum degree}, MATCH
  Commun. Math. Comput. Chem \textbf{81} (2019), no.~2, 393--404.

\bibitem{mallion1974graphical}
R.~B. Mallion, A.~J. Schwenk, and N.~Trinajsti{\'c}, \emph{Graphical study of
  heteroconjugated molecules}, Croatica Chemica Acta \textbf{46} (1974), no.~3,
  171--182.

\bibitem{mallion1974graph}
R.~B. Mallion, N.~Trinajsti{\'c}, and A.~J. Schwenk, \emph{Graph theory in
  {C}hemistry-{G}eneralisation of {S}achs' {F}ormula}, Zeitschrift f{\"u}r
  Naturforschung A \textbf{29} (1974), no.~10, 1481--1484.

\bibitem{shelash2020pseudospectrum}
H.~Shelash and A.~Shukur, \emph{Pseudospectrum {E}nergy of {G}raphs}, Iranian
  Journal of Mathematical Chemistry. \textbf{11} (2020), no.~2, 83--93.

\bibitem{wilson2015introduction}
R.J. Wilson, \emph{Introduction to {G}raph {T}heory}, Pearson Higher Ed, 2015.

\bibitem{zhou2020lower}
Q.~Zhou, D.~I. Wong, and D.~Q. Sun, \emph{A lower bound for graph energy},
  Linear and Multilinear Algebra \textbf{68} (2020), no.~8, 1624--1632.

\end{thebibliography}
\bibliographystyle{amsplain}
\end{document}